\documentclass[12pt,amstex]{article}

\usepackage[usenames]{color}

\usepackage{amscd}
\usepackage{amsmath}
\usepackage{amssymb}
\usepackage{graphicx}

\newtheorem {theo} {\bf Theorem} [section]
\newtheorem {prop} [theo] {\bf Proposition}
\newtheorem {cory} [theo] {\bf Corollary}
\newtheorem {lem} [theo] {\bf Lemma}
\newtheorem {defn} [theo] {\bf Definition}

\newtheorem {rem} [theo] {\bf Remark}

\newcommand{\qed}{\nopagebreak\hfill{\vrule width6pt height6pt depth0pt}}
\newenvironment{proof}[1]{\vspace{0.4 cm}\noindent{\em #1.}}{\qed\par
\vspace{\topsep}\vspace{\partopsep}}

\newcommand{\be}{\begin{eqnarray}}
\newcommand{\ee}{\end{eqnarray}}
\newcommand{\benn}{\begin{eqnarray*}}
\newcommand{\eenn}{\end{eqnarray*}}
\newcommand{\bse}{\begin{equation}}
\newcommand{\ese}{\end{equation}}
\newcommand{\bsenn}{\begin{displaymath}}
\newcommand{\esenn}{\end{displaymath}}
\newcommand{\logand}{\;\;{\rm and }\;\;}

\newcommand{\logif}{\;\;{\rm if }\;\;}

\newcommand{\R}{\mathbb{R}}

\newcommand{\Z}{\mathbb{Z}}
\begin{document}

\title{A Remarkable Summation Formula, Lattice Tilings, and Fluctuations}
\markright{Remarkable Summation}
\author{J. J. P. Veerman, L. S. Fox, P. J. Oberly}

\maketitle

\noindent
\section*{Abstract.}
\begin{small} We derive and prove an explicit formula for the sum of the fractional parts of certain
geometric series. Although the proof is straightforward, we have been unable to locate any reference
to this result.

This summation formula allows us to efficiently analyze the average behavior of certain common
nonlinear dynamical systems, such as the angle-doubling map, $x \mapsto 2x$ modulo 1. In particular, one
can use this information to analyze how the behavior of individual orbits deviates from the global average
(called fluctuations). More generally, the formula is valid in $\R^m$, where expanding maps give rise to
so-called number systems. To illustrate the usefulness in this setting, we compute the fluctuations of
a certain map on the plane.
\end{small}

\section{Introduction.}
\label{chap:intro}
\setcounter{figure}{0} \setcounter{equation}{0}

\noindent
The aim of this note is to state and prove an explicit formula that expresses the sum of the fractional parts
of certain geometric series and to demonstrate its usefulness. This formula (Theorem \ref{thm:summation})
and some of its corollaries are discussed in Section \ref{chap:main}. In this introduction, we discuss
an important, and intuitive, special case (Lemma \ref{lem:intro}).
In Sections \ref{chap:fluctuations} and \ref{chap:tileflucs}, we will use these results to better understand the
statistical behavior of certain (expanding) dynamical systems.

Consider the map $T:[0,1)\rightarrow [0,1)$ defined by
\bse
T:x\rightarrow 2x \mod 1\,.
\label{eq:expandingmap}
\ese
This is often called the \emph{angle doubling map}, because it describes the dynamics restricted to the unit
circle in the complex plane of $z\mapsto z^2$, \cite{devaney}. The action of that map is that it doubles
the angle. Thus repeated applications of $T$ tend to separate nearby initial points exponentially fast.
For that reason, this map serves as a paradigm for \emph{chaotic}\footnote{The word \emph{chaotic} can be given a precise meaning \cite{banks}. However, we will not need it.} dynamical systems.
A convenient way to study the behavior of orbits under $T$ is to write the initial condition $x$ in base 2:
\bse
x = 0.d_1d_2d_3\cdots = \sum_{i=1}^\infty\,2^{-i}d_i\,,
\label{eq:expansion}
\ese
where each $d_i$ is either 0 or 1. It is easy to see that
\bsenn
T(x)= 0.d_2d_3d_4 \cdots \quad \textrm{and so} \quad T^n(x)=0.d_{n+1}d_{n+2}d_{n+3}\cdots \,.
\esenn

In nonlinear dynamical systems such as this one, it is as a rule too much to ask for exact solutions. One
often settles for studying \emph{average} behavior, and, as we shall see later, the distribution of the
deviations from the average.

So let us do some averaging. One checks that $\frac 37$, $\frac 67$, $\frac 57$ form an orbit of period 3 under $T$.
The three points of the periodic orbit sum to 2, and so their average equals $\frac 23$. Now let us look at this
in base 2. It is a simple exercise to show that the three points are
\bsenn
0.\overline{011} \quad  0.\overline{110} \quad 0.\overline{101}\,.
\esenn
The overbar is used to indicate periodicity. So $0.\overline{011}$ indicates $0.011011011\cdots$. What we observe is
that the average of the orbit equals the average of the number of ones in the binary expansion of the initial condition
$x$. This is not a coincidence! We leave it to the reader to check the following example of period 5:
\bsenn
\frac{5}{31}\quad \frac{10}{31} \quad \frac{20}{31} \quad \frac{9}{31} \quad \frac{18}{31} \,.
\esenn
In binary notation, this becomes
\bsenn
0.\overline{00101} \quad 0.\overline{01010} \quad 0.\overline{10100} \quad 0.\overline{01001} \quad 0.\overline{10010}\,.
\esenn
Again, one checks that the average of the orbit equals the average number of ones in the binary expansion.
For periodic orbits, this is actually pretty easy to see by formally summing the binary expressions
of all the periodic points. We leave that as an exercise, and move on to a stronger statement.

We introduce the notation $p_n=\sum_{i=1}^n\,d_i$ to denote the running number of ones in the binary
representation of $x$ in the digits 1 through n (to the right of the decimal point).
Also, to get our notation closer to that of the general case, we write $\left\{2x\right\}$ for $T(x)$.
The braces are standard notation for ``fractional part". So $T^i(x)$ becomes $\left\{2^ix\right\}$.

We will now give an informal proof of the following, perhaps surprising, fact that the average of $\left\{2^ix\right\}$ converges if and only if the average of $p_n$ does, and that the limits are equal.
A more general statement with a complete proof will follow in Section \ref{chap:main}.

\begin{lem} If the base 2 expansion of $x\in[0,1)$ does not end in all ones, then
\bse
\lim_{n\rightarrow\infty}\left(\frac 1n\,\sum_{i=0}^{n-1}\,\left\{2^ix\right\}-\frac{p_n}{n}\right)=0\,.
\label{eq:average}
\ese
\label{lem:intro}
\end{lem}

\begin{proof}
For simplicity, let us set $x=\sum_{i=1}^{n}\, 2^{-j}d_j$.
We first establish the following equality:
\bse
\sum_{i=0}^{n-1}\, \left\{2^i x\right\} = p_n-x\,.
\label{eq:sum}
\ese
The crux is the observation that we can do the summation for each digit separately,
and only after that add up the results. So for fixed $1\leq j\leq n$ this gives
\bsenn
\sum_{i=0}^{n-1}\, \left\{2^i2^{-j}d_j\right\}= \sum_{i=0}^{j-1}\,2^i 2^{-j} d_j=
\frac{2^j-1}{2-1}\;2^{-j}d_j= d_j-d_j2^{-j}\,.
\esenn
(Note that in the first summation the fractional part of $2^i 2^{-j} d_j$ is zero if $i\geq j$.)
Summing over $j$ from 1 to $n$ now gives equation \eqref{eq:sum}. The lemma then follows
by taking an average of equation \eqref{eq:sum}, and then taking the limit as $n$ tends to infinity.
\end{proof}

There is a problem with taking the limit in this proof if the expansion of $x$ ends in all ones,
because the function $x\rightarrow \{x\}$ is not continuous as $x$ approaches 1 from the left. This will
be overcome in the more general treatment in Section \ref{chap:main}. For now, suffice it to say that the
exceptional set is countable and so will not affect the statistical reasoning in Section \ref{chap:fluctuations}.

Equation \eqref{eq:sum} has an amusing corollary. We know that $\sum_{i=0}^{n-1}\,2^ix=(2^n-1)x$.
Since $\left[2^i x\right]=2^ix-\left\{2^i x\right\}$, we now obtain the following.
\bse
\sum_{i=0}^{n-1}\,\left[2^ix\right]=2^nx-p_n\,.
\label{eq:integersum}
\ese

The reason we present these basic facts is twofold. First, the map $T$ is probably one of the most studied
maps in mathematics, and yet, to our surprise, with one exception \cite{veer},
we do not know of any explicit mention of these simple facts, including Lemma \ref{lem:intro}.
Second, and no less important, these identities can be useful. In fact, \eqref{eq:sum} was employed in
\cite{veer} to greatly simplify the characterization of the statistical properties of orbits of a one
parameter family of piecewise linear circle maps (see Section \ref{chap:fluctuations}).

For the interested reader, we remark that the study of the average behavior of \emph{typical} orbits of a dynamical system is a branch of mathematics, called ergodic theory, with far-reaching consequences
for physics \cite{AA}. It is natural and interesting to investigate the deviations from average behavior; such deviations are called \emph{fluctuations} and have a history of hundreds of years.
In the classical theory, one assumes or establishes statistical independence of certain events ---
coin tosses say --- and from this one derives Gaussian distributions for the fluctuations. Some of the
mathematical history can be found in the short, but delightful book \cite{Kac}. The application of
these ideas to physics form the basis for the sub-discipline of physics called
\emph{statistical physics} (see \cite{Reichl}).

Since the 1980s, however, Tsallis and others have noted that in systems with correlations over long distances --- which make the system less chaotic --- statistical independence might not hold
(see \cite{tsallis}[Chapters 5 and 7]). They proposed that fluctuations in such systems might
typically have distributions that are not Gaussian,
but a generalization thereof (\cite{tsallis}[Section 3.2] and references therein). This is currently
a very active area of research. In general, these fluctuations are very hard to analyze, in particular
if they do not have a Gaussian distribution. However, recently two low-dimensional systems with
non-Gaussian fluctuations have been analyzed. It is intriguing that one of them \cite{bountis}
seems to conform to the theory proposed, while the other does not \cite{veer}.

\section{Fluctuations in One Dimension.}
\label{chap:fluctuations}
Let $f$ be a map from the unit interval to itself. Consider
\bsenn
S(n,x):=\sum_{i=0}^{n-1} f^i(x)\,.
\esenn
The \emph{fluctuations} are defined as the deviations of the partial sums from the average.
Let us assume, for simplicity, that the average of $f$ (over $x$) equals $1/2$. Then by
fluctuations, we mean the distribution of
\bsenn
S(n,x)-\frac n2\,.
\esenn
as $x$ varies uniformly over all initial conditions and parameters \emph{while holding $n$ fixed}.
We are interested in establishing whether a properly rescaled version of these fluctuations tends to a
limiting distribution as $n$ tends to infinity.

An important generic example is the angle doubling map $T$ in the introduction.
\bsenn
S(n,x):=\sum_{i=0}^{n-1} \{2^i x\} \,.
\esenn
Almost all $x \in [0,1)$ have a unique base 2 expansion. Thus Lemma \ref{lem:intro} tells us what
the result is. For large $n$, the distribution of $S$ tends to that of $\sum_{i=1}^{n}\,d_i$ because the contribution of $\left\{2^n x\right\} - x$ is negligible.
The resulting distribution is of course the binomial one
\bsenn
\textrm{Prob}(S\leq k)=2^{-n}\sum_{x=0}^k \binom{n}{x}\,,
\esenn
which has mean $\mu=\frac n2$ and standard deviation $\sigma=\frac 12 \sqrt{n}$. It is well-known that for
large $n$, the binomial distribution is increasingly well approximated by the Gaussian distribution
$\frac{1}{\sigma\sqrt{2\pi}}\,\int_{-\infty}^k\,e^{-\frac 12 \left(\frac{x-\mu}{\sigma}\right)^2}\,dx$. With $\mu$ and $\sigma$
as given above, we conclude that
\bse
\lim_{n\rightarrow\infty}\, \textrm{Prob}\left(\frac{2(S-\frac n2)}{\sqrt{n}}\leq k\right)=
\frac{1}{\sqrt{2\pi}}\,\int_{-\infty}^k\,e^{-\frac 12 x^2}\,dx \,.
\label{eq:normaldistr}
\ese
The derivative with respect to $k$
is the \emph{probability density} $\frac{1}{\sqrt{2\pi}}\,e^{-\frac 12 k^2}$
associated with the fluctuations (rescaled by $1/\sqrt{n}$).

In the above case, the conclusion is that the distribution of the fluctuations --- upon appropriate
rescaling --- is Gaussian. This is expected for systems like $T$ that are expanding and therefore chaotic.
However, the interesting cases, in view of Tsallis' theory mentioned in the introduction, are those for
which the fluctuations have a non-Gaussian distribution that can be analyzed exactly. Since there are very
few of these, any explicit examples are very useful to gain intuition. We briefly present the example from \cite{veer}.

\begin{figure}[pbth]
\begin{center}
\includegraphics[width=3.5cm]{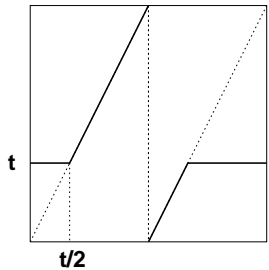}
\caption{Construction of a family of circle maps $f_t(x)$ given in \eqref{eq:ftx} based on truncations of the
angle doubling map. The horizontal axis is $x$. The vertical axis is $f_t(x)$. }
\label{fig:flatspotmap}
\end{center}
\end{figure}
We look at a family of maps that are considerably less chaotic, see Figure \ref{fig:flatspotmap}.
\bse
f_t(x)=\left\{ \begin{matrix} t   & \textrm{ for } 0\leq x\leq \frac t2\\[0.1cm]
                              2x  & \textrm{ for } \frac t2\leq x\leq \frac 12\\[0.1cm]
                              2x-1  & \textrm{ for } \frac 12\leq x\leq \frac{1+t}{2}\\[0.1cm]
                              t  & \textrm{ for } \frac{1+t}{2}\leq x\leq 1 \end{matrix}    \right.
\label{eq:ftx}
\ese
Following \cite{bountis, veer}, in this case, we consider
\bsenn
S(n,t,x):=\sum_{i=0}^{n-1} f_t^i(x) \,,
\esenn
and we define the fluctuations as the deviations of the partial sums from the average over $x$ \underline{and} $t$.
From symmetry considerations, one can see that this average is $1/2$. It is well-known that each of the maps
$f_t$ has dynamics similar (in fact, semi-conjugate) to a pure rotation \cite{veer} (and references therein).
That means that for each fixed value of $t$, each iterate of the map advances on average by a constant
$\rho(t)\in[0,1]$, called the rotation number. It turns out that $\rho(t)$ is an
increasing function of $t$. Points under $f_t$ rotate faster if $t$ is large. Thus iterations for different
$t$ drift apart linearly in $n$. To get something that converges to a distribution, we will need to
shift $S$ by the average $1/2$ and rescale by $1/n$. So we consider
\bse
\frac 1n\left(S(n,t,x)-\frac n2\right)=\frac 1n\,\sum_{i=0}^{n-1} \left(f_t^i(x)-\frac 12\right)\,,
\label{eq:sumcirclemap}
\ese
where now both $t$ and $x$ are uniformly distributed over $[0,1)$ while $n$ is fixed but large.
The hope is that the distribution of this tends (as $n\rightarrow \infty$) to a fixed non-trivial distribution.

\begin{figure}[!pbth]
\begin{center}
\includegraphics[height=4.0cm,width=5.0cm]{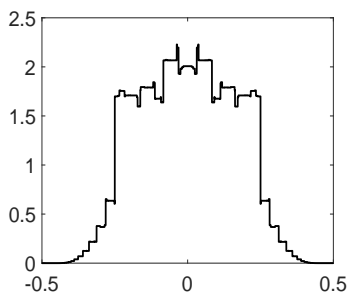}
\caption{The probability density function for the fluctuations of the family of maps $f_t$.}
\label{fig:ghostbuster}
\end{center}
\end{figure}
The computations are now substantially more complicated, but an important step is again a summation to
which Lemma \ref{lem:intro} is applied. Here we can only give a very cursory idea how the
proof goes. For the details, we refer to \cite{veer}.
We need to average $S(n,t,x)$ in \eqref{eq:sumcirclemap} over $x$ and $t$.
It is sufficient to consider only cases where the rotation is rational (the rest has
measure zero). So choose $t$ in the interval $I_{p/q}$ of values such that the rotation number of $f_t$
is equal to $\frac pq$ where $\gcd(p,q)=1$. With that choice, $x=t$ (see Figure \ref{fig:flatspotmap})
is a stable periodic orbit of period $q$ which attracts all initial conditions. From that one can derive
that it essentially determines the sum in \eqref{eq:sumcirclemap} for all $t\in I_{p/q}$. So, for large $n$
\bsenn
\frac 1n\left(S\left(n, t,x\right)-\frac n2\right)\approx
\frac 1q\,\sum_{i=0}^{q-1} \left(\left\{2^i t\right\}-\frac 12\right)\,.
\esenn
Lemma \ref{lem:intro} is applied to evaluate this sum and average over $t\in I_{p/q}$
and over $x$. Finally, we obtain the overall probability density associated with the fluctuations
as a sum over the rationals $\frac pq$ between 0 and 1. It can be computed numerically to
arbitrary precision and has a remarkable (and very non-Gaussian) appearance as can be appreciated in Figure
\ref{fig:ghostbuster}.

\section{A Summation Formula.}
\label{chap:main}

In this section, we consider expansions in $\R^m$ that are generalizations of those in \eqref{eq:expansion}. The number $2$ in that equation
is replaced by a matrix $A$ which is expanding (i.e. all eigenvalues have
modulus greater than 1) and whose entries are integers so that $A\Z^m$ forms a sublattice of $\Z^m$.
The determinant of $A$ is $\pm q$ for some integer $q$ greater than 1. This means that $A$ maps
the unit cube $[0,1]^m$ to a parallelotope (an $m$ dimensional ``parallelogram") of volume $q$.
Call two elements of $\Z^m$ equivalent if they differ by an element of $A\Z^m$. There are precisely $q$ equivalence classes of points in $\Z^m$. A set $D$ is a \emph{standard digit set} if it contains
exactly one element for each equivalence class. In number theoretic terms, $D$ corresponds to a
\emph{complete set of residues modulo $A$}. Without loss of generality, we require that $0\in D$.
This is what is called a standard number system. We summarize the definition here.

\begin{defn} A standard number system $(A,D)$ in $\R^m$ is given by an expanding matrix with
integer coefficients $A$, a digit set $D$ of cardinality $q=|\det A|$ containing exactly one element
of each coset of $A\Z^m$ in $\Z^m$, one of which is the origin.
\label{def:numbersystem}
\end{defn}

Consider vectors in $\R^m$ of the following form
\bse
x = A^{\ell}d_{-\ell}+\cdots A^0d_0+A^{-1}d_1+A^{-2}d_2+\cdots=
\sum_{i=-\ell}^\infty\,A^{-i}d_i\,,
\label{eq:expansion-gen}
\ese
where the \emph{digits} $d_i$ in the above expression are chosen from a standard digit set $D$ and
$d_{-\ell}\ne 0$. Since $A$ is expanding, $A^{-1}$ is well-defined with eigenvalues of modulus
strictly smaller than 1, and so the sum in \eqref{eq:expansion-gen} converges exponentially.

\begin{defn} For $x\in \R^m$ and an expansion ${\boldsymbol d}$ given by
$x=\sum_{i=-\ell}^\infty\,A^{-i}d_i$, define the \emph{integer part} $[x]$ and the
\emph{fractional part} $\{x\}$ as
\bse
[x({\boldsymbol d})] = \sum_{i=-\ell}^0\,A^{-i}d_i \quad \logand \quad
\{x({\boldsymbol d})\} = \sum_{i=1}^\infty\,A^{-i}d_i \,,
\label{eq:IandF}
\ese
Denote the set of \emph{fractional parts} by $\cal F$ and the set of \emph{integral parts} by $\cal I$.
\label{def:int-frac-part}
\end{defn}

\noindent
{\bf Remark.} It is important to bear in mind that the definition above is different from the
usual definition of integral and fractional parts in that these depend on the expansion
${\boldsymbol d}$ of $x$. For example, with this definition, $0.111\cdots$ (in base 2) has
fractional part 1 and integral part 0, while $1.000\cdots$ has
fractional part 0 and integral part 1.

\vskip .1in
It turns out that for the number systems in $\R^m$ that interest us, these cases have measure
zero. Since we aim to do statistical calculations, we can safely neglect them. Some more details
are given in Section \ref{chap:tilings}. Because of this, and for notational convenience,
we drop the dependence on the expansion from our notation, and write $\{x\}$ and $[x]$ from now on.

Now we get to our main results.

\begin{theo}[Fractional Part Summation Formula] If $(A,D)$ is a standard number system, then for
any $x\in \R^m$ with a base $A$ expansion $\sum_{i=-\ell}^{\infty}\,A^{-i}d_i$
\bsenn
(A-I)\,\sum_{i=0}^{n-1}\,\{A^ix\}=\sum_{i=1}^{n}\,d_i + \left\{A^nx\right\} - \{x\}\,.
\esenn
\label{thm:summation}
\end{theo}

\begin{proof} First, split up $\{x\}$ as follows
\bse
\{x\}=\sum_{i=1}^{n}\,A^{-i}d_i + \sum_{i=n+1}^\infty\,A^{-i}d_i := x_n + y_n \,.
\label{eq:xn-yn}
\ese
As in Lemma \ref{lem:intro}, one observes that the summation can be carried out for each digit
separately, and then the result can be added up. So let us start with the digits in $x_n$.
For a fixed $j\in\{1,\cdots ,n\}$, using the fact that $A-I$ is invertible, we get
\bsenn
\sum_{i=0}^{n-1}\,\{A^i A^{-j}d_j\}= \sum_{i=0}^{j-1}\,A^i A^{-j}d_j= (A-I)^{-1}\,(A^j-I)\,A^{-j}d_j\,.
\esenn
In the first equality, $A^i A^{-j}d_j$ are integer vectors for $i\geq j$, so their fractional parts are
zero. Thus we can change the upper limit of the sum in order to remove $\{\cdot \}$. The second equality
is a geometric series. Multiplying by $A-I$ and then summing over $j\in\{1,\cdots ,n\}$ gives
\bsenn
(A-I)\sum_{i=0}^{n-1}\,\{A^i x_n\}= \sum_{j=1}^{n}\,d_j - x_n\,.
\esenn
Since $\{A^iy_n\}=A^iy_n$ for all $0<i\leq n$, the remaining summation is also a geometric series:
\bsenn
(A-I)\sum_{i=0}^{n-1}\,\{A^iy_n\}=(A-I)\sum_{i=0}^{n-1}\,A^iy_n=(A^n-I)y_n\,.
\esenn
Adding the last two displayed equations gives
\bsenn
(A-I)\sum_{i=0}^{n-1}\,\{A^i(x_n+y_n)\}=\sum_{j=1}^{n}\,d_j + A^ny_n- x_n - y_n\,.
\esenn
and noting that $x_n+y_n=\{x\}$ and $A^ny_n=\left\{A^nx\right\}$ gives the result.
\end{proof}

\begin{cory} If $(A,D)$ is a standard number system, then for
any $x\in \R^m$ with a base $A$ expansion $\sum_{i=-\ell}^{\infty}\,A^{-i}d_i$
\bsenn
\lim_{n\rightarrow \infty}\left((A-I)\,\frac 1n \,\sum_{i=0}^{n-1}\,\{A^ix\}-\frac 1n \,\sum_{i=1}^{n}\,d_i\right) =0\,.
\esenn
\label{cor:average}
\end{cory}

\begin{proof} Note that in Theorem \ref{thm:summation}, $\{x\}$ and $\left\{A^nx\right\}$ are fractional numbers. Thus their modulus is less than some a priori bound. The corollary follows immediately upon
dividing by $n$ and taking a limit as $n\rightarrow \infty$.
\end{proof}

Theorem \ref{thm:summation} also has a curious corollary analogous to equation \ref{eq:integersum}.

\begin{cory} If $(A,D)$ is a standard number system, then for
any $x\in \R^m$ with a base $A$ expansion $\sum_{i=-\ell}^{\infty}\,A^{-i}d_i$
\bsenn
(A-I)\,\sum_{i=0}^{n-1}\,\left[A^ix\right]=\left[A^nx\right]-[x]-\sum_{i=1}^{n}\,d_i \,.
\esenn
\label{cory:summation}
\end{cory}

\begin{proof} We note that $\sum_{i=0}^{n-1}\,\left[A^ix\right]$ equals $\sum_{i=0}^{n-1}\,A^ix-\sum_{i=0}^{n-1}\,\left\{A^ix\right\}$. The first of these is a geometric series
and the second is given by Theorem \ref{thm:summation}.
\end{proof}

It is instructive to look at a few examples. First, consider $\frac 54$ expanded in base 2 as
$1.01000\cdots$ and take $n=4$ in Theorem \ref{thm:summation}. Since $A-I=1$, the left hand side gives
$\frac 14 + \frac 12 + 0 + 0$. The right hand side gives $0+1+0+0$ for the sum of the digits and
$0-\frac 14$ for $\{A^nx\} - \{x\}$. Thus
\bsenn
\left(\frac 14 +\frac 12 +0 +0\right)= (0+1+0+0) +0-\frac 14 \,.
\esenn
On the other hand, if we expand $\frac 54$ as $1.00111\cdots$, then the same calculation gives
\bsenn
\left(\frac 14 +\frac 12 +1 +1\right)= (0+0+1+1) +1-\frac 14 \,.
\esenn
Note that in this case $\{x\}$ does not have its usual meaning (see Definition \ref{def:int-frac-part}).

We do the same computations for Corollary \ref{cory:summation}. First, we consider the expansion
$1.01000\cdots$ and again $n=4$. The corollary gives
\bsenn
\left(1+2+5+10\right)=20-1-(0+1+0+0) \,.
\esenn
The other expansion, $1.00111\cdots$, gives the following equality
\bsenn
\left(1+2+4+9\right)=19-1-(0+0+1+1) \,.
\esenn
Corollary \ref{cor:average} is also easy to check.

\section{Self-Affine Tilings.}
\label{chap:tilings}

It is natural to ask whether at least some of these computations involving fluctuations can be done in
higher dimensions. The answer is yes. We will discuss such an example in Section \ref{chap:tileflucs}.
First we need some background.

To start with a familiar example, in the standard binary (or decimal) expansion, the set of fractional
parts $\cal F$ form a compact \emph{tile}, $[0,1]$, and the \emph{tiling set} is $\Z$. So, $[0,1]+\Z$
covers $\R$, and any two translations of $[0,1]$ by distinct elements of $\Z$ have an intersection of
measure zero. Interestingly, $\Z$ is \emph{not} the same as the set of integer parts ${\cal I}$ (which
are the non-negative integers). However, the collection of differences in ${\cal I}$ does contain $\Z$.
So, in this case, ${\cal F}$ is a tile and the tiling set are the elements of ${\cal I}-{\cal I}$. This
turns out to be a common pattern. It is, for example, easy to see that the same is true for the usual
decimal (base 10) expansion.

We now give some results for systems $(A,D)$ as in Definition \ref{def:numbersystem}.
Denote the set of all differences in $\cal I$ by
\bsenn
\Delta := \cal I - \cal I \,.
\esenn

\begin{defn}
If $(A,D)$ is a standard number system, then we say that ${\cal F}+\Delta$ is a \emph{tiling} (of $\R^m$) if
$\R^m = {\cal F}+\Delta$, but any two distinct translates of ${\cal F}$ by elements of ${\Delta}$ have intersection of
measure zero.
\label{def:tile}
\end{defn}

The following result gives some criteria for when a number system gives rise to a tiling.

\begin{prop} Given the number system $(A,D)$ as in Definition \ref{def:numbersystem}, then
${\cal F}+\Delta$ is a tiling if and only if $\Delta$ is a lattice \cite{HSV}. This is always the case
in dimension one \cite{groch}. In dimension 2 and 3, it includes all cases where $D$ has two elements
\cite{HSV}.
\label{prop:tiling}
\end{prop}
There are more general criteria for when number systems in dimension 2 and greater give rise
to tilings, however, they are more complicated to state (but see \cite{CHR, LW2}).

It may come as a surprise that $\Delta$ is not always a lattice. So here is a counter-example \cite{LW}.
\bse
A= \begin{pmatrix} 2&1\\0&2 \end{pmatrix}\;, \quad D=\left\{\begin{pmatrix} 0\\0 \end{pmatrix},
\begin{pmatrix} 3\\0 \end{pmatrix}, \begin{pmatrix} 0\\1 \end{pmatrix}, \begin{pmatrix} 3\\1
\end{pmatrix}\right\}\,.
\label{eq:lagwang}
\ese
\begin{figure}[!ht]
\begin{center}
\includegraphics[width=12.0cm]{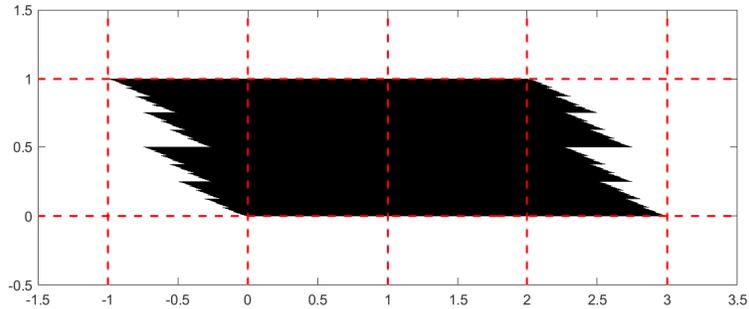}
\caption{The tile generated by the number system of equation \eqref{eq:lagwang}.}
\label{fig:lagwang}
\end{center}
\end{figure}
The corresponding tile can be seen in Figure \ref{fig:lagwang}. It turns out that ${\cal F}$ is a tile
with tiling set $3\Z\times \Z$ but that tiling set is not equal to $\Delta$.
We also point out that a lattice can mean a sublattice of $\Z$ (or $\Z^m$). For example, in $\R$,
\bsenn
A=3\;, \quad D=\left\{-5,0,20 \right\} \,,
\esenn
gives rise to a tile whose tiling set is $\Delta=5\Z$. So the Lebesgue measure of $\cal F$ (a single tile)
is 5.

We briefly discuss some general properties of tilings.
First, from \eqref{eq:IandF} one sees that a tile $\cal F$ satisfies
\bse
{\cal F}=A^{-1}\left({\cal F}+D\right)\,.
\label{eq:self-affine}
\ese
This leads one to define a map $\tau$ on the space of a priori bounded, compact sets $Z$ (with
an appropriate topology), namely
\bsenn
\tau(Z)=A^{-1}\left(Z+D\right)\,.
\esenn
One can prove that $\tau$ is a contraction on a complete metric space and thus has a unique
fixed point \cite{Hut}. That fixed point, of course, is $\cal F$ (by \eqref{eq:self-affine}). The
dynamical system $\tau$ is usually called an \emph{iterated function system}.

On the other hand, \eqref{eq:self-affine} also --- and quite literally --- says that $\cal F$ consists of
$q$ affine copies ${\cal F}_i=A^{-1}({\cal F}+\delta_i)$ of itself, where $D=\cup_{i=1}^q\{\delta_i\}$.
This allows us to define an expanding map
$T$ from $\cal F$ to itself by setting
\bse
T(x)= Ax-\delta_i \;\logif \; x\in {\cal F}_i\,.
\label{eq:dynsys}
\ese
This is the generalization of the angle doubling map in equation \eqref{eq:expandingmap}.
There is a very simple algorithm to generate an expansion for any given vector $x\in \R^m$.
It is the same algorithm that works in the case that $T(x)=\{2x\}$.
Namely, first multiply by $A^{-k}$ (if necessary) to ensure that $x$ is a fractional number,
i.e. so that $y=A^{-k}x \in {\cal F}$. Now if $y\in {\cal F}_{i}$, then set $d_1:=\delta_{i}$.
Then compute $T(y)=Ay-d_1$. Next, if $T(y)\in {\cal F}_{j}$, then set $d_2:=\delta_{j}$.
Compute $T^2(y)=AT(y)-d_2$, and so on. This computes an expansion of $y$.
At the end we multiply back by $A^k$ to get the expansion of $x$.

Note that ambiguity arises only in the case where $T^k(y)$ falls in the intersection of 2 or more of the ${\cal F}_i$.
\bsenn
T^k(y)=\sum_{i=1}^\infty\,A^{-i}d_i = \sum_{i=1}^\infty\,A^{-i}d_i' \,.
\esenn
This happens if its image under $AT^k(y)$ lies in the intersection of two distinct tiles.
In a standard number system, these have measure zero according to Definition \ref{def:numbersystem}.
For each $k$, we obtain a set a measure zero where $d_k$ is not unique.
Thus the exceptions form a countable union of measure zero sets, and thus have measure zero. In fact, we can
do better. For example, if $A:\R^m\rightarrow \R^m$ is a similarity, the Hausdorff dimension $\partial {\cal F}$
can be computed and is always strictly less than $m$ \cite{veer-mex}. The exceptional set,
being a countable union of copies, has the same dimension.

\section{Tiles and Fluctuations.}
\label{chap:tileflucs}

In Section \ref{chap:fluctuations}, we computed the fluctuations for the map given by \eqref{eq:expandingmap}. Here we do the same, but now for the more general map defined in \eqref{eq:dynsys}.
However, to keep things simple, we consider one example: a 2-dimensional tile with 2 digits. There are exactly 6
possibilities for the characteristic polynomial of $A$, namely \cite{HSV}: $x^2\pm 2$, $x^2\pm x+2$,
and $x^2\pm 2x +2$. Let us consider the last polynomial and set
\bse
A= \begin{pmatrix} -1&1\\-1&-1 \end{pmatrix}\;, \quad D=\left\{\begin{pmatrix} 0\\0 \end{pmatrix},
\begin{pmatrix} 1\\0 \end{pmatrix}\right\} \,.
\label{eq:2digit}
\ese
Denote these digits simply by $0$ and $d$. The matrix $A$ has eigenvalues $1\pm i$ and so is a similarity
($\sqrt{2}$ times a rotation by $\frac{-3\pi}{8}$). Therefore the resulting tile $\cal F$ is self-similar
as opposed to merely self-affine.

\begin{figure}[pbth]
\begin{center}
\includegraphics[width=9.0cm]{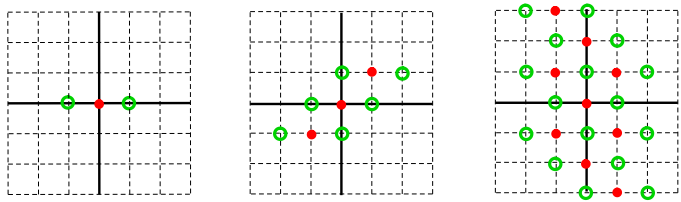}
\caption{Starting from the left with the origin in red. Add the non-zero differences in $D$. In the next figure
to the right, in red $A$ times the previous collection of points, then in green the non-zero differences
in $D$ are added to each red point.}
\label{fig:lattice}
\end{center}
\end{figure}
First, we check that this number system indeed gives the tiling ${\cal F}+\Z^2$. By Proposition \ref{prop:tiling}, we need to make sure that $\Delta$ equals $\Z^2$. The easiest way is to do this
recursively as illustrated in Figure \ref{fig:lattice}.
Start with the origin (left, in red). Then add the non-zero elements of $(D-D)$ (left, in green).
Multiply the result by $A$ to get $A(D-D)$ (middle, red) and add $(D-D)$ (middle, green) to get
$A(D-D)+(D-D)$. Repeating the procedure to get $A^2(D-D)+A(D-D)+(D-D)$ gives the figure on the right.
Readers should be able to convince themselves that this recursion obtains $\Delta=\Z^2$.

We conclude that ${\cal F} +\Z^2$ is a tiling. The corresponding tile can be seen in Figure
\ref{fig:2digit}. It is known as the Heighway dragon and has a long and storied history (see \cite{Tab}
and references therein).
\begin{figure}[!ht]
\begin{center}
\includegraphics[width=8.0cm]{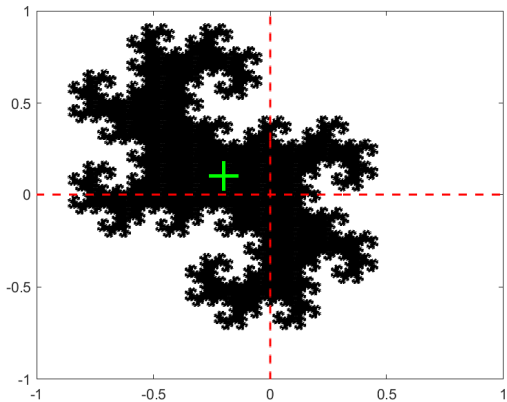}
\caption{The tile generated by the number system of equation \eqref{eq:2digit}. The center of mass
is located at the center of the green cross.}
\label{fig:2digit}
\end{center}
\end{figure}

Recall that we are interested in the fluctuations
\bsenn
S(n,x) - n \overline{x}=\sum_{i=0}^{n-1} \left\{A^ix\right\}\;-n\overline{x}\,,
\esenn
of the map \eqref{eq:dynsys}, where $\overline{x}$ is its average --- or center of mass. So we compute the average of
$S(n,x)$ with  with $x$ uniformly distributed in $\cal F$. Again, Theorem \ref{thm:summation} gives the answer. For large $n$, the contribution of $\left\{A^nx\right\}-x$ is negligible, and so
\bse
S(n,x)\approx (A-I)^{-1}\sum_{i=1}^nd_i=p_n (A-I)^{-1}d\,,
\label{eq:fluc-heighway}
\ese
where $p_n$ is the number of times $d$ occurs in the first $n$ digits of $x$. The location $\overline{x}$ of the
center of mass of this tile follows immediately, since we only have to take the average of $p_n$, which is
$1/2$. With \eqref{eq:fluc-heighway} this implies the next result.

\begin{prop} The center of mass for the 2-digit tile defined by \eqref{eq:2digit}, is given by
\bsenn
\overline{x}=\frac{1}{2} (A-I)^{-1}d \,.
\esenn
\label{prop:centermass}
\end{prop}

\vskip -0.2in\noindent
For the Heighway dragon defined by \eqref{eq:2digit}, we get (see Figure \ref{fig:2digit})
\bsenn
\overline{x}:=\frac 12 (A-I)^{-1}d=\frac{1}{10} \begin{pmatrix} -2 & -1 \\ 1 & -2 \end{pmatrix}\begin{pmatrix} 1\\ 0 \end{pmatrix} =
\frac{1}{10} \begin{pmatrix} -2\\ 1 \end{pmatrix}\,.
\esenn

Together with equation \eqref{eq:fluc-heighway} this says that $S$ is approximately equal to $2p_n\overline{x}$.
This implies another, and more remarkable, corollary, namely that, for large $n$, the fluctuations in this
2-digit system lie on the line $t\overline{x}$ where $t\in\R$ ! The limiting distribution itself is easy to figure
out since the distribution of $p_n$ is exactly the same as the one given in \eqref{eq:normaldistr}.
This then gives the following proposition.

\begin{prop} The distribution of $\;2(S-n \overline{x})\,n^{-1/2}$ tends to the following 1-dimensional
distribution as $n\rightarrow\infty$
\bsenn
\mathrm{Prob}\left(x\in \left\{t\overline{x} : t\leq \frac k2\right\}\right)=
\frac{1}{\sqrt{2\pi}}\,\int_{-\infty}^k\,e^{-\frac 12 u^2}\,du \,.
\esenn
\label{prop:fluctuations}
\end{prop}

To illustrate this counter-intuitive result numerically, we plotted the fluctuations in Figure
\ref{fig:flattening} for $n=15$ (left) and $n=50$ (right) together with the line $t\overline{x}$. In each case,
we generated $10^4$ random
binary strings (of length 15 and 50, respectively), computed $S$ for each string, and then plotted
$2(S-n \overline{x})n^{-1/2}$. The `flattening' of the distributions is clearly observable.
\begin{figure}[!ht]
\begin{center}
\includegraphics[width=6.2cm]{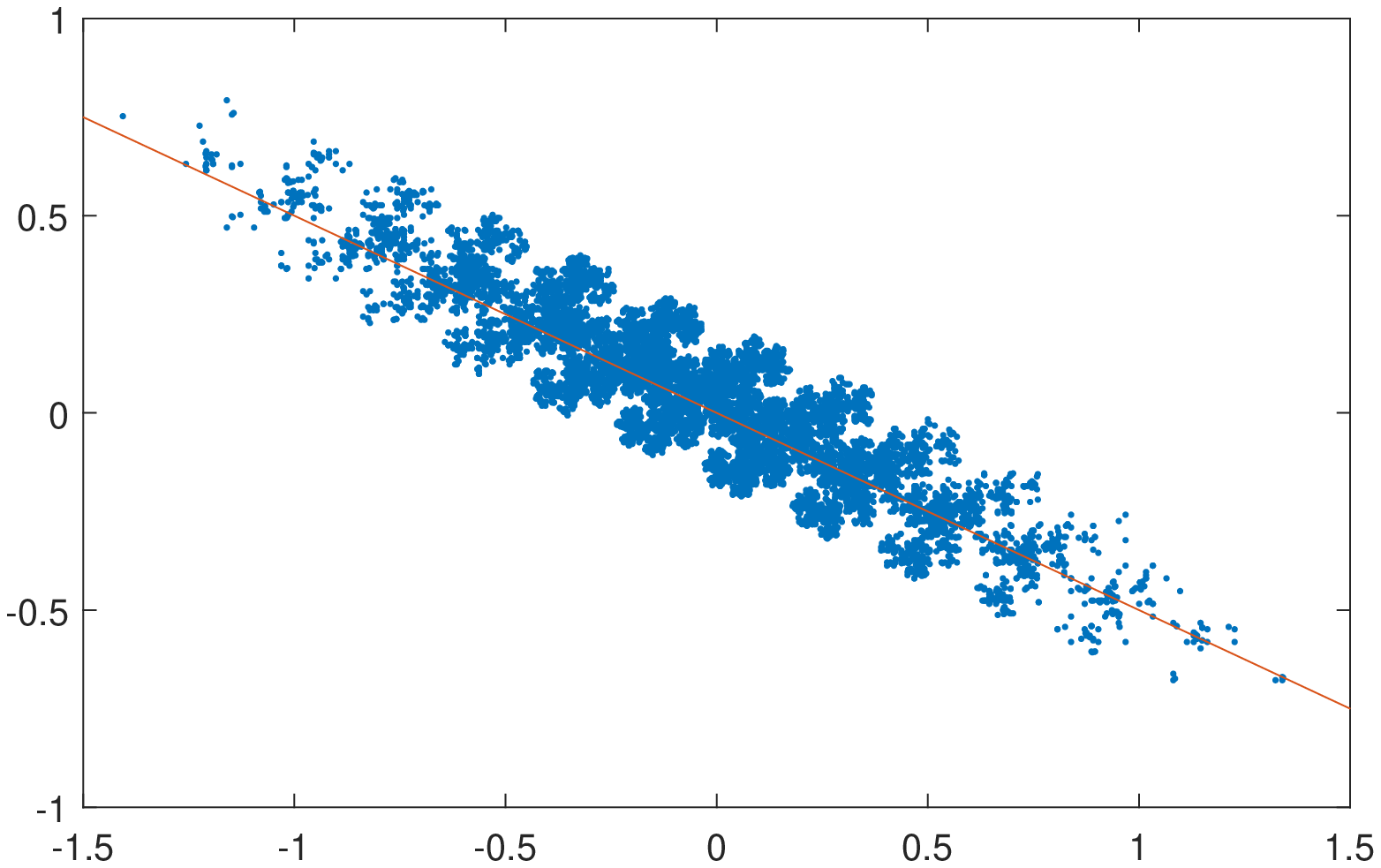}
\includegraphics[width=6.2cm]{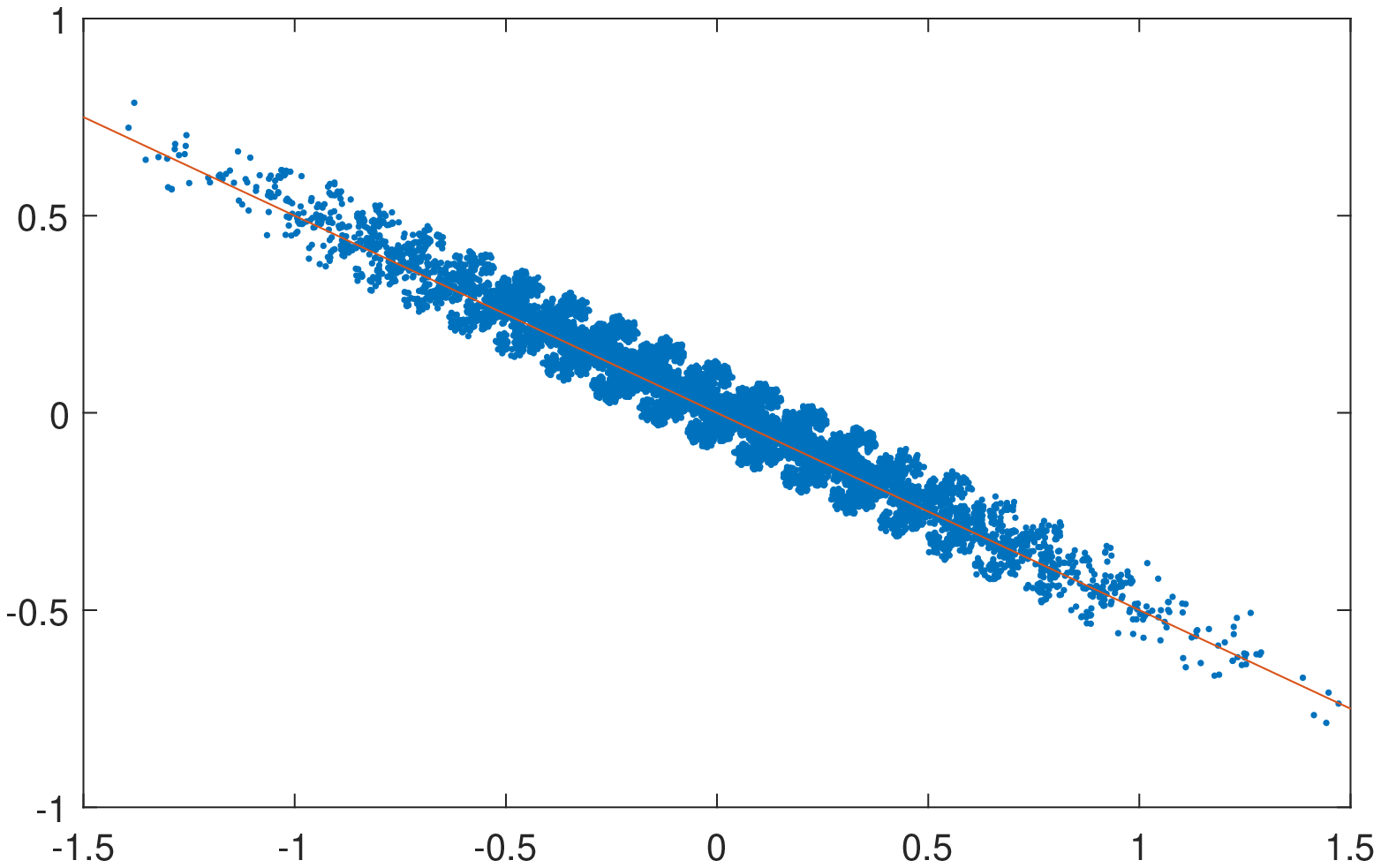}
\caption{The distributions of $\frac{2(S-n \overline{x})}{\sqrt{n}}$ for $n=15$ on the left and for
$n=50$ on the right.}
\label{fig:flattening}
\end{center}
\end{figure}

\noindent
{\bf Acknowledgment.}
We are indebted to an anonymous referee for many useful remarks that improved the paper, in particular
one observation that ended up strengthening Theorem \ref{thm:summation}.


%

\vspace{\fill}
\end{document}